\documentclass{amsart}

\usepackage[]{fontenc}
\usepackage[latin1]{inputenc}
\usepackage{geometry}
\usepackage{fancyhdr}
\usepackage{setspace}
\usepackage{pstricks}
\usepackage{pst-node}
\usepackage{pst-plot}
\usepackage[all]{xy}
\usepackage{amssymb}

\usepackage{amsfonts}
\usepackage{eufrak}
\usepackage{amsmath}
\usepackage{amsthm}
\usepackage{a4wide}
\usepackage{textcomp}
\usepackage{epsfig}

\newtheorem{thm}{Theorem}[section]
\newtheorem{defi}[thm]{Definition}

\newtheorem{conj}[thm]{Conjecture}
\newtheorem{prop}[thm]{Proposition}
\newtheorem{rmk}[thm]{Remark}
\newtheorem{lemma}[thm]{Lemma}

\begin{document}

\title[A Morse estimate for translated points]
{A Morse estimate for translated points of contactomorphisms of spheres and projective spaces}

\author{Sheila Sandon}

\begin{abstract}
\noindent
A point $q$ in a contact manifold $(M,\xi)$ is called a \textit{translated point} for a contactomorphism $\phi$ with respect to some fixed contact form if $\phi(q)$ and $q$ belong to the same Reeb orbit and the contact form is preserved at $q$. In this article we discuss a version of the Arnold conjecture for translated points of contactomorphisms and, using generating functions techniques, we prove it in the case of spheres (under a genericity assumption) and projective spaces.
\end{abstract}

\maketitle

\section{Introduction}\label{intro}

Let $(M,\xi)$ be a cooriented contact manifold, with a fixed contact form $\alpha$. Given a contactomorphism $\phi$ of $M$ we denote by $g: M \rightarrow \mathbb{R}$ the function satisfying $\phi^{\ast}\alpha=e^g\alpha$. A point $q$ of $M$ is called a \textbf{translated point} of $\phi$ (with respect to the contact form $\alpha$) if $q$ and $\phi(q)$ belong to the same Reeb orbit and if moreover the contact form $\alpha$ is preserved at $q$, i.e. $g(q)=0$. The following theorem is the main result of this article.

\begin{thm}\label{main}
Consider the unit sphere $S^{2n-1}$ in $\mathbb{R}^{2n}$ with its standard contact form $\alpha=xdy-ydx$, and the projective space $\mathbb{R}P^{2n-1}$ seen as the quotient of $S^{2n-1}$ by the antipodal action of $\mathbb{Z}_2$, with the induced contact form.
\begin{enumerate}
\renewcommand{\labelenumi}{(\roman{enumi})}
\item Every generic contactomorphism of $S^{2n-1}$ which is contact isotopic to the identity has at least $2$ translated points.
\item Every contactomorphism of $\mathbb{R}P^{2n-1}$ which is contact isotopic to the identity has at least $2n$ translated points.
\end{enumerate}
\end{thm}

Note that $2$ and $2n$ are the minimal number of critical points of a function defined respectively on $S^{2n-1}$ and $\mathbb{R}P^{2n-1}$. More generally we can thus wonder whether the following is true.

\begin{conj}\label{C}
Let $\phi$ be a contactomorphism of a contact manifold $\big(M,\xi=\text{ker}(\alpha)\big)$. Assume that $M$ is compact and that $\phi$ is contact isotopic to the identity. Then the number of translated points of $\phi$ is at least equal to the minimal number of critical points of a function on $M$.
\end{conj}

This problem, that was already introduced in \cite{mio4}, can be considered as a contact version of the Arnold conjecture for fixed points of Hamiltonian symplectomorphisms. Recall that the Arnold conjecture says that the number of fixed points of a Hamiltonian symplectomorphisms $\varphi$ of a compact symplectic manifold $(W,\omega)$ is at least equal to the minimal number of critical points of a function on $W$. This conjecture was posed in the 60's and since then has played a fundamental role in the development of symplectic topology. Although the generic version has been proved for all symplectic manifolds, the general conjecture still remains open.\\
\\
Note that contactomorphisms do not necessarily have fixed points: for example, since the Reeb vector field never vanishes, the Reeb flow for small times does not have any fixed point. On the other hand, at least for $\mathcal{C}^1$-small (and even $\mathcal{C}^0$-small) contactomorphisms translated points always exist. In fact, as we will now discuss, Conjecture \ref{C} is true in this case. The proof is completely analogous to the proof of the Arnold conjecture in the $\mathcal{C}^1$-small and $\mathcal{C}^0$-small cases (for which we refer for example to \cite{MS}): the $\mathcal{C}^1$-small case only relies on Weinstein's neighborhood theorem for Legendrian submanifolds, while the $\mathcal{C}^0$-small case is deeper and relies on the existence of generating functions quadratic at infinity for contact deformations of the 0-section in the 1-jet bundle.\\
\\
We consider the contact product of $M$ with itself, i.e. the manifold $M\times M\times\mathbb{R}$ endowed with the contact structure given by the kernel of the 1-form $A=e^{\theta}\alpha_1-\alpha_2$ where $\theta$ is the coordinate on $\mathbb{R}$ and $\alpha_1$ and $\alpha_2$ are the pullback of $\alpha$ by the projections on the first and second factors respectively. Note that the Reeb vector field of $A$ is $R_A=(0,-R_{\alpha},0)$ where $R_{\alpha}$ is the Reeb vector field of $\alpha$ on $M$. We define the diagonal and the graph of $\phi$ in $M\times M\times\mathbb{R}$ to be the Legendrian submanifolds $\Delta=\{(q,q,0)\;|\; q\in M\}$ and $\text{gr}_{\phi}=\{\big(q,\phi(q),g(q)\big)\;|\; q\in M\}$ where $g$ is the function defined by $\phi^{\ast}\alpha=e^g\alpha$. Then translated points of $\phi$ correspond to Reeb  chords between $\Delta$ and $\text{gr}_{\phi}$. Indeed, a point $q$ of $M$ is a translated point for $\phi$ if and only if the corresponding point of the graph is of the form $\big(q,\phi(q),g(q)\big)=\big(q,(\varphi^{R_{\alpha}})_{t_0}(q),0\big)$ for some $t_0$, where $(\varphi^{R_{\alpha}})_t$ is the Reeb flow of $\alpha$. Note that this point is in the same Reeb orbit as the point $(q,q,0)$ of the diagonal. The problem of counting translated points of $\phi$ is thus reduced to the problem of counting Reeb chords between its graph and the diagonal. By Weinstein's neighborhood theorem for Legendrian submanifolds we know that a neighborhood of $\Delta$ in $M\times M\times\mathbb{R}$ is contactomorphic, by a contactomorphism that also preserves the contact form, to a neighborhood of the 0-section in the 1-jet bundle $J^1\Delta$ (see \cite[Theorem 2.2.4]{AH}). If we assume that $\phi$ is $\mathcal{C}^1$-small then its graph is contained in this neighborhood of $\Delta$, and so it corresponds to a Legendrian submanifold $\Gamma_{\phi}$ of $J^1\Delta$, which is moreover a section. Recall that all Legendrian sections of a 1-jet bundle are the 1-jet of a function on the base. We have thus that $\Gamma_{\phi}=j^1f=\{(q,df(q),f(q))\;|\; q\in \Delta\}$ for some function $f$ on $\Delta$. Translated points of $\phi$ correspond to Reeb chords between $\Gamma_{\phi}$ and the 0-section and hence to critical points of $f$. Since $\Delta$ is diffeomorphic to $M$, Conjecture \ref{C} follows under our $\mathcal{C}^1$-smallness assumption. The same conclusion can also be obtained if we assume that $\phi$ is the time-1 map of a $\mathcal{C}^0$-small contact isotopy. In this case the graph of $\phi$ corresponds to a contact deformation of the 0-section in $J^1\Delta$, and so we can apply Chekanov's theorem \cite{C} on existence of generating functions to obtain Morse or cup-length estimates for the number of translated points.\\
\\
Note that, in contrast to the symplectic case, the discussion above is not sufficient to show that Conjecture \ref{C} is sharp\footnote{I am very grateful to Will Merry for pointing this out to me.}. In the symplectic case, fixed points of a $\mathcal{C}^1$-small symplectomorphism $\varphi$ of a symplectic manifold $W$ correspond exactly to the critical points of the induced function $f$ on $W$ (fixed points of $\varphi$ correspond to intersections of the graph with the diagonal, hence to intersections of $df$ with the 0-section and hence to critical points of $f$). In the contact case, on the other hand, we do not have a 1-1 correspondence between translated points of a $\mathcal{C}^1$-small contactomorphism $\phi$ of a contact manifold $M$ and critical points of the associated function $f$: every critical point of $f$ corresponds to a translated point of $\phi$, but there could be translated points of $\phi$ that do not correspond to critical points of $f$. Indeed, translated points of $\phi$ correspond to Reeb chords between the graph of $\phi$ and the diagonal in $M\times M\times\mathbb{R}$, but not all such Reeb chords are necessarily contained in a Weinstein neighborhood of the diagonal. Thus, they might not correspond to Reeb chords between $j^1f$ and the 0-section in $J^1M$ and so to critical points of $f$. In other words, there might be a translated point $q$ of $\phi$ such that $\phi(q)$ belongs to a small neighborhood of $q$ (since $\phi$ is $\mathcal{C}^1$-small) but all Reeb chords connecting $q$ and $\phi(q)$ go out of this small neighborhood. This clearly cannot happen for $S^{2n-1}$ and $\mathbb{R}P^{2n-1}$ because the Reeb orbits are the Hopf fibers. Thus, Theorem \ref{main} is sharp. We do not know whether Conjecture \ref{C} is sharp in general.\\
\\
As discussed in \cite{mio4}, translated points are a special case of leafwise coisotropic intersections. Given a contactomorphism $\phi$ of $\big(M\,,\,\xi=\text{ker}(\alpha)\big)$ contact isotopic to the identity we can lift it to a Hamiltonian symplectomorphism $\widetilde{\phi}$ of the symplectizaton $\big(M\times \mathbb{R}\,,\,\omega=d(e^{\theta}\alpha)\big)$ by defining $\widetilde{\phi}(q,\theta)=\big(\phi(q),\theta-g(q)\big)$ where as usual $g$ is the function defined by $\phi^{\ast}\alpha=e^g\alpha$. Then translated points of $\phi$ correspond to leafwise intersections of $M\equiv M\times\{0\}$ and $\widetilde{\phi}(M)$. Recall that, given a coisotropic submanifold $Q$ of a symplectic manifold $(W,\omega)$ and a Hamiltonian symplectomorphism $\varphi$ of $W$, the leafwise intersections problem looks for points $q$ of $Q$ such that $\widetilde{\phi}(q)$ belongs to the same leaf as $q$ with respect to the characteristic foliation of $Q$. Leafwise intersections need not always exist in general, so when studying this problem it is really necessary to pose some condition either on the Hamiltonian symplectomorphism $\varphi$ or on the ambient symplectic manifold. A typical condition in the literature is a smallness condition for $\varphi$ with respect to the Hofer norm. As far as I know, the condition of $\varphi$ being the lift to the symplectization of a contactomorphism was considered for the first time in \cite{mio4}. We refer to this article for more details on the relation of our question with the problem of leafwise intersections. Here we will just illustrate the difference of the two approaches in the special case of the sphere $S^{2n-1}$.\\
\\
We regard $S^{2n-1}$ as the unit sphere in $\mathbb{R}^{2n}$. Then its characteristic foliation as a hypersurface of $\mathbb{R}^{2n}$ coincides with the Reeb foliation  as a contact manifold. By applying Theorem A of Albers-Frauenfelder \cite{AF2} to the special case of $S^{2n-1}$ we get the following result.

\begin{thm}[\cite{AF2}]\label{AF}
Let $\varphi$ be a Hamiltonian symplectomorphism of $\mathbb{R}^{2n}$ with Hofer norm smaller than the displacement energy of $S^{2n-1}$. Then there exists a leafwise intersection for $\varphi$. If moreover $\varphi$ is generic then there are at least 2 leafwise intersections.
\end{thm}

In the more general context in which this theorem is stated, the Hamiltonian symplectomorphism $\varphi$ is required to be smaller than the minimal period of a Reeb orbit, and the conclusion in the generic case is that the number of leafwise intersections is bounded below by the sum of the Betti numbers of the contact manifold. In the case of the sphere the minimal period of a Reeb orbit coincides with the displacement energy, thus Theorem \ref{AF} is sharp: there is no hope to improve it by improving the smallness bound for $\varphi$. However part (i) of our Theorem \ref{main} gives a different class of Hamiltonian symplectomorphisms of $\mathbb{R}^{2n}$ for which the same result is true: those Hamiltonian symplectomorphisms that are the lift to $\mathbb{R}^{2n}$ of a contactomorphism of $S^{2n-1}$. As far as I understand Theorem \ref{main}(i) and Theorem \ref{AF} are independent, and there is no clear way to deduce one from the other. Note that given the lift to $\mathbb{R}^{2n}$ of a contactomorphism of $S^{2n-1}$ we can cut off the Hamiltonian, without changing the leafwise intersections, so that it makes sense to look at its Hofer norm. However for an arbitrary contactomorphism of $S^{2n-1}$ this Hofer norm is in general not small.\\
\\
In contrast to the fact that $S^{2n-1}$ is displaceable in $\mathbb{R}^{2n}$, the following theorem of Ekeland and Hofer \cite{EH} shows that $S^{2n-1}$ is not displaceable by a $\mathbb{Z}_2$-equivariant Hamiltonian isotopy, with respect to the action of $\mathbb{Z}_2$ on $\mathbb{R}^{2n}$ given by $(x,y)\mapsto(-x,-y)$. In fact, they prove that it is not even leafwise displaceable.

\begin{thm}[\cite{EH}]\label{EH}
Let $\varphi$ be the time-1 map of a $\mathbb{Z}_2$-equivariant Hamiltonian isotopy of $\mathbb{R}^{2n}$. Then there exists a leafwise intersection for $\varphi$.
\end{thm}

Note that, since there is no smallness assumption for the Hamiltonian symplectomorphism $\varphi$, this theorem implies in particular existence of at least one translated point for every contactomorphism of $\mathbb{R}P^{2n-1}$ contact isotopic to the identity. Part (ii) of our Theorem \ref{main} can thus be seen as a generalization of this conclusion.\\
\\
The proof of Theorem \ref{main} uses generating functions techniques in the setting developed by Th\'{e}ret \cite{Th1,Th2} and Givental \cite{Giv}. In \cite{Th1,Th2} Th\'{e}ret studies Hamiltonian isotopies of complex projective space and in particular proves the Arnold conjecture in this case, i.e. existence of at least $n$ fixed points for every Hamiltonian symplectomorphism of $\mathbb{C}P^{n-1}$ (a result that had already been obtained by Fortune and Weinstein \cite{FW,F} with different techniques). Th\'{e}ret obtains this result by lifting the problem to $\mathbb{R}^{2n}$ via the Hopf fibration $S^{2n-1}\rightarrow\mathbb{C}P^{n-1}$, and then by studying generating functions for Hamiltonian symplectomorphism of $\mathbb{R}^{2n}$ that are the lift of Hamiltonian symplectomorphism of $\mathbb{C}P^{n-1}$. A similar approach was also used by Givental in \cite{Giv} to define the non-linear Maslov index for contactomorphisms of $\mathbb{R}P^{2n-1}$. Our proof of Theorem \ref{main} is very much based on these two articles.\\
\\
Given a contactomorphism $\phi$ of $S^{2n-1}$ contact isotopic to the identity we lift it to a Hamiltonian symplectomorphism $\Phi$ of $\mathbb{R}^{2n}$. $\mathbb{R}_+$-lines of fixed points of $\Phi$ correspond to \textit{discriminant points} of $\phi$, i.e. points $q$ of $S^{2n-1}$ with $\phi(q)=q$ and $g(q)=0$ where as usual $g$ is the function defined by $\phi^{\ast}\alpha=e^g\alpha$. Following Th\'{e}ret \cite{Th1,Th2} we show that $\Phi$ has a generating function $F:\mathbb{R}^{2n}\times\mathbb{R}^{2N}\rightarrow\mathbb{R}$ which is homogeneous of degree 2 and hence induces a function $f$ on the unit sphere $S^{2n+2N-1}$. Critical points of $F$ corresponds to fixed points of $\Phi$. They come in $\mathbb{R}_+$-lines and always have critical value 0. To find translated points of $\phi$ we consider the contact isotopy $a_t\circ\phi$, where $a_t$ is the negative Reeb flow $z\mapsto e^{-2\pi i t}z$, and a 1-parameter family $F_t$ of generating functions for the lifted Hamiltonian isotopy of $\mathbb{R}^{2n}$. Translated points of $\phi$ correspond to the union for all $t$ in $[0,1]$ of the critical points with critical value 0 of the induced functions $f_t$ on the unit sphere of the domain of $F_t$. By studying the change in the topology of $f_t^{\phantom{t}-1}(0)$ for $t$ varying in $[0,1]$ one sees that in the generic case $\phi$ must have at least $2$ translated points. In the case of a contactomorphism $\phi$ of $\mathbb{R}P^{2n-1}$ we apply the same argument to a lift to $S^{2n-1}$. We prove that the corresponding generating functions $F_t$ in this case are $\mathbb{Z}_2$-invariant and thus induce functions $f_t$ on a real projective space. We then show that this difference in the topology of the domain of the generating functions implies the existence of $2n$ instead of only $2$ translated points. The main tool in this $\mathbb{Z}_2$-equivariant case is the cohomological index for subsets of real projective space, that was introduced  by Fadell and Rabinowitz \cite{FR} and also used in \cite{Giv}.\\
\\
This article is organized as follows. In the next section we give some preliminaries on contactomorphisms, and in particular we explain how to lift our problem to $\mathbb{R}^{2n}$. In Section \ref{gf} we describe generating functions in this context. Our treatment is mostly based on \cite{Th1,Th2}. Finally in Sections \ref{S} and \ref{RP} we prove respectively parts (i) and (ii) of Theorem \ref{main}.

\subsection*{Acknowledgments} This work was done during my postdoc at the Laboratoire Jean Leray of Nantes, supported by an ANR GETOGA fellowship. I thank Vincent Colin and Paolo Ghiggini for discussions and support, and the whole Laboratoire Jean Leray for providing such a pleasant and stimulating ambiance. I am very grateful to the organizers of the Conference on Geometrical Methods in Dynamics and Topology in Hanoi, the Conference of the Trimester on Contact and Symplectic Topology in Nantes, the GESTA 2011 Conference in Castro Urdiales and the Workshop on Conservative Dynamics and Symplectic Geometry in Rio de Janeiro for giving me the opportunity to participate to these events and to present there the results contained in this article. I also thank Carlos Moraga and Fran\c{c}ois Laudenbach for convincing me that I needed the non-degeneracy condition in part (i) of Theorem \ref{main}. Finally I would like to take the opportunity to thank again my PhD supervisor Miguel Abreu because it is especially now that I realize how much I am benefiting from his choice to concentrate his efforts as a supervisor in teaching me the most difficult thing a supervisor can teach, i.e. how to find and recognize a good mathematical problem.

\section{Preliminaries}\label{bg}

Let $(M^{2n-1},\xi)$ be a contact manifold, i.e. an odd dimensional manifold $M$ endowed with a maximally non-integrable distribution $\xi$ of hyperplanes. We will always assume that $\xi$ is cooriented, so that it can be written as the kernel of a 1-form $\alpha$ such that $\alpha\wedge (d\alpha)^{n-1}$ is a volume form. A diffeomorphism $\phi$ of $M$ is called a contactomorphism if it preserves the contact distribution $\xi$ and its coorientation, i.e. if $\phi^{\ast}\alpha=e^g\alpha$ for some function $g: M \rightarrow \mathbb{R}$. An isotopy $\phi_t$, $t\in[0,1]$, of $M$ is called a contact isotopy if every $\phi_t$ is a contactomorphism.\\
\\
An important example of a contact isotopy is given by the \textit{Reeb flow} $(\varphi^{R_{\alpha}})_t $ associated to a contact form $\alpha$, i.e. the flow of the vector field $R_{\alpha}$ which is defined by the two conditions $\iota_{R_{\alpha}}d\alpha=0$ and $\alpha(R_{\alpha})=1$. The definition of the Reeb flow is a special case of a more general construction that associates to every time-dependent function $h_t:M\rightarrow \mathbb{R}$ a contact isotopy $\phi_t$ of $M$ starting at the identity. The isotopy $\phi_t$ is defined to be the flow of the vector field $X_t$ determined by the two conditions $\iota_{X_t}d\alpha=dh_t(R_{\alpha})\alpha-dh_t$ and $\alpha(X_t)=h_t$. One can check that $\phi_t$ is a contact isotopy, and moreover that it is the unique contact isotopy satisfying the condition $\alpha(X_t)=h_t$. The time-dependent function $h_t:M\rightarrow \mathbb{R}$ is called the \textit{(contact) Hamiltonian} of the contact isotopy $\phi_t$. \\
\\
Contact manifolds are intimately related to symplectic manifolds. A symplectic manifold is an even dimensional manifold $W$ endowed with a closed non-degenerate 2-form $\omega$. A diffeomorphism $\Phi$ of $W$ is called a symplectomorphism if $\Phi^{\ast}\omega=\omega$. Similarly to the contact case, every time-dependent function $H_t:W\rightarrow\mathbb{R}$ induces a symplectic isotopy $\Phi_t$ of $W$, which is defined to be the flow of the vector field $X_t$ determined by the condition $\iota_{X_t}\omega=dH_t$. However, in contrast with the contact case, not all symplectic isotopies can be obtained in this way. Those isotopies that that are obtained by a time-dependent function via this construction are called \textit{Hamiltonian isotopies}.\\
\\
A way to relate symplectic and contact manifolds is given by associating to a contact manifold $(M,\xi)$ its \textit{symplectization}. If we fix a contact form $\alpha$ for $\xi$ then the symplectization of $M$ is defined to be the manifold $SM=M\times\mathbb{R}$ endowed with the symplectic form $\omega=d(e^{\theta}\alpha)$, where $\theta$ is the $\mathbb{R}$-coordinate. Every contactomorphism $\phi$ of $M$ can be lifted to a symplectomorphism $\Phi$ of $SM$ by the formula $\Phi(q,\theta)=\big(\phi(q)\,,\,\theta-g(q)\big)$ where $g$ is the function on $M$ satisfying $\phi^{\ast}\alpha=e^g\alpha$. Given a contact isotopy $\phi_t$ of $M$ generated by the contact Hamiltonian $h_t:M\rightarrow \mathbb{R}$, the lift $\Phi_t$ is a Hamiltonian isotopy of $SM$ which is generated by the Hamiltonian function $H_t: SM \rightarrow\mathbb{R}$, $H_t(q,\theta)=e^{\theta}h_t(q)$.\\
\\
We can now define the central objects of study of this paper.

\begin{defi}
Let $(M,\xi)$ be a contact manifold with a fixed contact form $\alpha$. Consider a contactomorphism $\phi$ of $M$ and let $g$ be the function satisfying $\phi^{\ast}\alpha=e^g\alpha$. A point $q$ of $M$ is called a \textbf{translated point} of $\phi$ (with respect to the contact form $\alpha$) if $q$ and its image $\phi(q)$ belong to the same Reeb orbit, and if moreover $g(q)=0$ (i.e. the contact form is preserved at $q$). A point $q$ of $M$ is called a \textbf{discriminant point} of $\phi$ if it is a translated point which is also a fixed point, i.e. if $\phi(q)=q$ and $g(q)=0$.
\end{defi}

Discriminant points of contactomorphisms were first studied by Givental in \cite{Giv}.\\
\\
Note that discriminant points of a contactomorphism $\phi$ of $M$ correspond to $\mathbb{R}$-lines of fixed points for the lifted symplectomorphism $\Phi$ of $SM$: $q$ is a discriminant point of $\phi$ if and only if $(q,\theta)$ is a fixed point of $\Phi$, for all $\theta\in\mathbb{R}$. We will say that a discriminant point $q$ of $\phi$ is \textit{non-degenerate} if there are no tangent vectors $X$ of $M$ at $q$ satisfying simultaneously $\phi_{\ast}X=X$ and $X(g)=0$. In terms of the lifted symplectomorphism $\Phi$ of $SM$, this condition is equivalent to saying that there should be no $X$ such that $\Phi_{\ast}(X,v)=(X,v)$ for some (and hence for all) $v\in T_{\theta}\mathbb{R}$. Here we regard $(X,v)$ as a tangent vector of $SM=M\times\mathbb{R}$ at any point $(q,\theta)$ above $q$. Note that we always have $\Phi_{\ast}(0,v)=(0,v)$ so, even if $q$ is a non-degenerate discriminant point of $\phi$, all points $(q,\theta)$ above it are degenerate fixed points of $\Phi$ (however, they are non-degenerate in the horizontal direction). A translated point $q$ of $\phi$ is said to be non-degenerate if $q$ is a non-degenerate discriminant point of $a_t\circ\phi$ where $a_t=(\varphi^{R_{\alpha}})_{-t}$ and $t$ is the smallest time for which $\phi(q)=(\varphi^{R_{\alpha}})_t(q)$.\\
\\
We will now specialize to the case of the sphere $S^{2n-1}$ endowed with its standard contact structure $\xi$. Recall that $\xi$ is defined to be the kernel of the restriction to $S^{2n-1}$ of the 1-form $\alpha=xdy-ydx$, where we regard $S^{2n-1}$ as the unit sphere in $\mathbb{R}^{2n}$ and where $(x,y)$ are the coordinates on $\mathbb{R}^{2n}$. Equivalently, after identifying $\mathbb{R}^{2n}$ with $\mathbb{C}^n$, $\xi$ is defined as the intersection of $TS^{2n-1}$ with $J(TS^{2n-1})$ where $J$ is the complex structure on $\mathbb{C}^n$. For the rest of the paper we will fix the contact form $\alpha$. Its Reeb flow is given by the Hopf fibration $z\mapsto e^{2\pi i t}z$.\\
\\
Let $\phi$ be a contactomorphism of $S^{2n-1}$, contact isotopic to the identity. As explained in the introduction, our goal is to show that $\phi$ has at least $2$ translated points. We will do this in Section \ref{S}. The strategy will be to count discriminant points of $a_t\circ\phi$ for $t\in[0,1]$, where $a_t$ is the negative Reeb flow $z\mapsto e^{-2\pi i t}z$. We will concentrate on discriminant points of contactomorphisms of $S^{2n-1}$ because they correspond to lines of fixed points for the lifted Hamiltonian symplectomorphism of $\mathbb{R}^{2n}$. As we will recall in the next section, fixed points of a symplectomorphism of $\mathbb{R}^{2n}$ correspond to critical points of its generating function, and can thus be detected by Morse theoretical arguments.  \\
\\
We will now explain how to lift our problem to $\mathbb{R}^{2n}$.\\
\\
Given a contactomorphism $\phi$ of $S^{2n-1}$ we can lift it to a symplectomorphism $\Phi$ of $\mathbb{R}^{2n}\equiv \mathbb{C}^n$ by the formula
\begin{equation}\label{lift}
\Phi(z)= \frac{|z|}{ e^{ \frac{1}{2}\,g( \frac{z}{|z|} ) }  }\,\phi(\frac{z}{|z|})
\end{equation}
where $g:S^{2n-1}\rightarrow S^{2n-1}$ is the function determined by $\phi^{\ast}\alpha=e^g\alpha$. Given a contact isotopy $\phi_t$ of $S^{2n-1}$ generated by the contact Hamiltonian $h_t: S^{2n-1} \rightarrow \mathbb{R}$, the lift $\Phi_t$ is the Hamiltonian isotopy of $\mathbb{R}^{2n}$ which is generated by the Hamiltonian function $H_t$ defined by $H_t(z)=|z|^2\,h_t(\frac{z}{|z|})$.

\begin{rmk}\label{diff}
The lift $\Phi$ is only defined on $\mathbb{R}^{2n}\smallsetminus \{0\}$, but we will extend it to $\mathbb{R}^{2n}$ by posing $\Phi(0)=0$. Then $\Phi$ is continuous everywhere and smooth on $\mathbb{R}^{2n}\smallsetminus \{0\}$. A similar observation also holds for the Hamiltonian function $H_t$.
\end{rmk}

Note that we can identify $\mathbb{R}^{2n}\smallsetminus \{0\}$ with the symplectization of $S^{2n-1}$ by the symplectomorphism $S^{2n-1}\times\mathbb{R} \rightarrow \mathbb{R}^{2n}\smallsetminus \{0\}$, $(q,\theta) \mapsto\sqrt{2}\,e^{\frac{\theta}{2}}q$. Under this identification, formula (\ref{lift}) reduces to the formula we gave above for the lift of a contactomorphism to the symplectization.\\
\\
Note that the lift $\Phi: \mathbb{R}^{2n}\rightarrow \mathbb{R}^{2n}$ of a contactomorphism $\phi$ of $S^{2n-1}$ is equivariant with respect to the radial $\mathbb{R}_+$-action, i.e. $\Phi(\lambda z)=\lambda \,\Phi(z)$ for all $\lambda\in\mathbb{R}_+$. Moreover the Hamiltonian function $H_t:\mathbb{R}^{2n}\rightarrow\mathbb{R}$ of the lift of a contact isotopy $\phi_t$ of $S^{2n-1}$ is homogeneous of degree 2, i.e. we have that $H_t(\lambda z) = \lambda^2 H_t(z)$ for all $\lambda\in\mathbb{R}_+$.\\
\\
As we will see in the next section, also generating functions of Hamiltonian isotopies of $\mathbb{R}^{2n}$ that are the lift of contact isotopies of $S^{2n-1}$ are homogeneous of degree $2$. In the rest of the article we will sometimes just write \textacutedbl homogeneous\textgravedbl  for \textacutedbl homogeneous of degree 2\textgravedbl.

\section{Homogeneous generating functions}\label{gf}

Let $\Phi$ be a Hamiltonian symplectomorphism of $\mathbb{R}^{2n}$. We will recall in this section how to associate to $\Phi$ a \textit{generating function} $F:\mathbb{R}^{2n}\times\mathbb{R}^{2N}\rightarrow\mathbb{R}$. The importance of this function is given by the fact that its Morse theory reflects the symplectic properties of $\Phi$. In particular, it has the crucial property that its critical points are in 1-1 correspondence with the fixed points of $\Phi$. After recalling how to construct a generating function for a Hamiltonian symplectomorphism $\Phi$ of $\mathbb{R}^{2n}$, we will show (following Th\'{e}ret \cite{Th2}) that if $\Phi$ is the lift of a contactomorphism $\phi$ of $S^{2n-1}$ then its generating function is homogeneous of degree 2. In the last part of this section we will then discuss two facts that will be crucial in the proof of Theorem \ref{main} . The first fact is monotonicity of generating functions (Lemma \ref{monotonicity}). The second is the fact that the negative Reeb flow $a_t$, $t\in[0,1]$, of $S^{2n-1}$ is generated by a 1-parameter family $A_t$ of quadratic forms such that the difference of the indices of $A_1$ and $A_0$ is equal to $2n$.\\
\\
Generating functions are functions associated to Lagrangian submanifolds of the cotangent bundle $T^{\ast}B$ of a smooth manifold $B$. Recall that an $n$-dimensional submanifold $L$ of a $2n$-dimensional symplectic manifold $(W,\omega)$ is called Lagrangian if $\omega|_{L}\equiv0$. Recall also that the cotangent bundle $T^{\ast}B$ has a canonical symplectic form $\omega_{\text{can}}$ given by $\omega_{\text{can}}=d(-\lambda_{\text{can}})$ where $\lambda_{\text{can}}$ is the tautological 1-form on $T^{\ast}B$: for a tangent vector $X$ at a point $\sigma$ of $T^{\ast}B$, $\lambda_{\text{can}}(X)=\sigma\big(\pi_{\ast}(X)\big)$ where $\pi$ is the projection $T^{\ast}B\rightarrow B$. Lagrangian sections of $T^{\ast}B$ are the graphs of closed 1-forms on $B$. Lagrangian sections that are Hamiltonian isotopic to the 0-section correspond to exact 1-forms and so they are the graph of the differential of a smooth function $f:B\rightarrow\mathbb{R}$. This function is then called a generating function for the corresponding Lagrangian section. To obtain generating functions of Lagrangian submanifolds of $T^{\ast}B$ that are not necessarily sections we use the following more general construction, which was introduced by H\"{o}rmander \cite{Hor}. Consider a fiber bundle $\pi: E\rightarrow B$. A function $F:E\rightarrow \mathbb{R}$ is called a generating function if $dF: E\longrightarrow T^{\ast}E$ is transverse to the fiber normal bundle $N_E:=\{\:(e,\mu)\in T^{\ast}E \;|\; \mu = 0 \;\text{on} \;\emph{ker}\,\big(d\pi\,(e)\big)\:\}$. Note that in this case the set $\Sigma_F=(dF)^{-1}(N_E)$ of fiber critical points of $F$ is a submanifold of $E$, of dimension equal to the dimension of $B$. Consider now the Lagrangian immersion $i_F:\Sigma_F\longrightarrow T^{\ast}B$, $e\mapsto \big(\pi(e),v^{\ast}(e)\big)$ where the element $v^{\ast}(e)$ of $T^{\phantom{\pi}\ast}_{\pi(e)}B$ is defined by $v^{\ast}(e)\,(X):=dF\,(\widehat{X})$ for $\widehat{X}$ any lift to $T_eE$ of the vector $X \in T_{\pi(e)}B$. $F$ is called a generating function for the image of $i_F$ in $T^{\ast}B$. Note that critical points of $F$ correspond under $i_F$ to intersections of $L$ with the 0-section.\\
\\
Consider now a Hamiltonian symplectomorphism $\Phi$ of $\mathbb{R}^{2n}$. Its graph $\{\big(z,\Phi(z)\big)\,,\,z\in\mathbb{R}^{2n}\}$ is a Lagrangian submanifold of $\overline{\mathbb{R}^{2n}}\times\mathbb{R}^{2n}=\big(\mathbb{R}^{2n}\times\mathbb{R}^{2n},-\omega\oplus\omega\big)$, and we will denote by $\Gamma_{\Phi}$ the Lagrangian submanifold of $T^{\ast}\mathbb{R}^{2n}$ which is the image of the graph of $\Phi$ by the symplectomorphism $\tau: \overline{\mathbb{R}^{2n}}\times\mathbb{R}^{2n}\rightarrow T^{\ast}\mathbb{R}^{2n}$, $(x,y, X, Y)\mapsto (\frac{x+X}{2},\frac{y+Y}{2}, Y-y, x-X)$. A generating function for $\Phi$ is by definition a generating function for the Lagrangian submanifold $\Gamma_{\Phi}$
of $T^{\ast}\mathbb{R}^{2n}$. Note that, since $\tau$ sends the diagonal of $\overline{\mathbb{R}^{2n}}\times\mathbb{R}^{2n}$ to the 0-section of $T^{\ast}\mathbb{R}^{2n}$, critical points of $F$ are in 1-1 correspondence with fixed points of $\Phi$.

\begin{prop}\label{ex_gf_ne}
Every Hamiltonian symplectomorphism $\Phi$ of $\mathbb{R}^{2n}$ has a generating function $F:\mathbb{R}^{2n}\times\mathbb{R}^{2N}\rightarrow\mathbb{R}$.
\end{prop}

This result is classical. We include the proof because it will be needed in the rest of the article. We first need the following composition formula, which is due to Th\'{e}ret \footnote{An analogous composition formula, which holds if we use the identification $(x,y,X,Y)\mapsto(y,X,x-X,Y-y)$ of $\overline{\mathbb{R}^{2n}}\times\mathbb{R}^{2n}$ with $T^{\ast}\mathbb{R}^{2n}$, was given by Chekanov \cite{C} and used for example in \cite{Tr}.}.

\begin{lemma}[Composition Formula \cite{Th2}]\label{cf}
Let $\Phi$ and $\Psi$ be Hamiltonian symplectomorphisms of $\mathbb{R}^{2n}$ with generating functions respectively $F:\mathbb{R}^{2n}\times\mathbb{R}^{2N}\rightarrow\mathbb{R}$ and  $G:\mathbb{R}^{2n}\times\mathbb{R}^{2N'}\rightarrow\mathbb{R}$. Then the function $F\sharp G: \mathbb{R}^{2n}\times (\mathbb{R}^{2n}\times\mathbb{R}^{2n}\times\mathbb{R}^{2N}\times\mathbb{R}^{2N'})\rightarrow\mathbb{R}$ defined by $F\sharp G \,(u;v,w,\mu,\eta)=F(u+w;\mu) + G(v+w;\eta) + 2<u-v,iw>$ is a generating function for $\Psi\circ\Phi$.
\end{lemma}

Here $<\cdot\,,\,\cdot>$ denotes the Euclidean scalar product on $\mathbb{C}^n\equiv\mathbb{R}^{2n}$.

\begin{proof}[Proof of Proposition \ref{ex_gf_ne}]
If $\Phi$ is $\mathcal{C}^1$-small then $\Gamma_{\Phi}$ is a Lagrangian section of  $T^{\ast}\mathbb{R}^{2n}$ and hence it has a generating function $F:\mathbb{R}^{2n}\rightarrow\mathbb{R}$. In the general case we consider a Hamiltonian isotopy $\Phi_t$ connecting $\Phi$ to the identity, and we subdivide it in $\mathcal{C}^1$-small pieces. We then apply Lemma \ref{cf} at every step.
\end{proof}

If $\Phi$ is a compactly supported Hamiltonian symplectomorphism of $\mathbb{R}^{2n}$ then it can be shown \cite{LS} that it has a (essentially unique \cite{V}) generating function $F:\mathbb{R}^{2n}\times\mathbb{R}^{2N}\rightarrow\mathbb{R}$ quadratic at infinity. In this paper we will not be interested in compactly supported Hamiltonian symplectomorphisms but in Hamiltonian symplectomorphisms of $\mathbb{R}^{2n}$ that are the time-1 map of Hamiltonian isotopies equivariant with respect to the radial action of $\mathbb{R}_+$. Recall from Section \ref{bg} that these Hamiltonian symplectomorphisms are exactly the lifts of contactomorphisms of $S^{2n-1}$ contact isotopic to the identity. We will now show that all such Hamiltonian symplectomorphisms have a generating function $F:\mathbb{R}^{2n}\times\mathbb{R}^{2N}\rightarrow\mathbb{R}$ which is homogeneous of degree 2. Recall that this means that $F(\lambda q;\lambda \mu)=\lambda^2F(q;\mu)$ for every $(q;\mu)\in\mathbb{R}^{2n}\times\mathbb{R}^{2N}$ and $\lambda\in\mathbb{R}_+$.

\begin{prop}\label{ex_gf}
Every Hamiltonian symplectomorphism $\Phi$ of $\mathbb{R}^{2n}$ which is the time-1 map of a $\mathbb{R}_+$-equivariant Hamiltonian isotopy has a generating function $F:\mathbb{R}^{2n}\times\mathbb{R}^{2N}\rightarrow\mathbb{R}$ which is homogeneous of degree $2$.
\end{prop}

\begin{proof}
Note that the composition formula of Lemma \ref{cf} preserves the property of being homogeneous of degree $2$: if $F$ and $G$ are homogeneous of degree $2$ then so is $F\sharp G$. Therefore it is enough to prove that if $\Phi$ is a $\mathcal{C}^1$-small $\mathbb{R}_+$-equivariant Hamiltonian symplectomorphism of $\mathbb{R}^{2n}$ then its generating function $F:\mathbb{R}^{2n}\rightarrow\mathbb{R}$ is homogeneous of degree $2$. Recall that $F: \mathbb{R}^{2n}\rightarrow\mathbb{R}$ being a generating function for $\Phi$ means that $\Gamma_{\Phi}\subset T^{\ast}\mathbb{R}^{2n}$ is the graph of $dF$. Note that $F$ is uniquely defined once we normalize it by requiring $F(0)=0$. We will now show that $F$ is homogeneous of degree $2$, i.e. $F(\lambda x, \lambda y)=\lambda^2F(x,y)$. For a fixed $\lambda\in\mathbb{R}_+$ consider the function $F_{\lambda}: \mathbb{R}^{2n}\rightarrow\mathbb{R}$, $F_{\lambda}(x,y)=\lambda^{-2}F(\lambda x, \lambda y)$. This function generates the image of $i_{F_{\lambda}}: \mathbb{R}^{2n} \rightarrow T^{\ast}\mathbb{R}^{2n}$,
$$i_{F_{\lambda}}(x,y) = (x,y,\frac{\partial F_{\lambda}}{\partial x},\frac{\partial F_{\lambda}}{\partial y} )=\big(x,y,\lambda^{-1}\frac{\partial F}{\partial x}(\lambda x, \lambda y), \lambda^{-1}\frac{\partial F}{\partial y}(\lambda x,\lambda y)\big)
$$
i.e. the subspace $\{\big(\lambda^{-1}x, \lambda^{-1}y, \lambda^{-1}\frac{\partial F}{\partial x}(x,y), \lambda^{-1}\frac{\partial F}{\partial y}(x,y)\big)\;|\;(x,y)\in\mathbb{R}^{2n}\}$. We denote by $\tau_{\lambda}: \mathbb{R}^{2n}\rightarrow\mathbb{R}^{2n}$ the map defined by $\tau_{\lambda}(x,y)=(\lambda x,\lambda y)$ and use the same notation also for the induced map on  $T^{\ast}\mathbb{R}^{2n}$. Then the Lagrangian submanifold of $T^{\ast}\mathbb{R}^{2n}$ generated by $F_{\lambda}$ is $\tau_{\lambda^{-1}}(\Gamma_{\Phi})$. But this coincides with $\Gamma_{\Phi}$ because $\Phi$ is $\mathbb{R_{+}}$-equivariant. Thus $F_{\lambda}$ generates $\Gamma_{\Phi}$, and since $F_{\lambda}(0)=F(0)=0$ we get that $F_{\lambda}=F$, i.e. $F$ is homogeneous of degree $2$.
\end{proof}

\begin{rmk}
By the way it is constructed and in view of Remark \ref{diff}, the generating function $F$ is not $\mathcal{C}^{\infty}$ everywhere: it is $\mathcal{C}^1$ with Lipschitz differential everywhere, and $\mathcal{C}^{\infty}$ outside the union of a finite number of hyperplanes intersecting the set $\Sigma_F$ of fiber critical points only at $0$. In particular, $F$ is $\mathcal{C}^{\infty}$ near its non-zero critical points.
\end{rmk}

Let $\phi$ be a contactomorphism of $S^{2n-1}$, contact isotopic to the identity, and consider a homogeneous generating function\footnote{This generating function is not unique, but it depends on the contact isotopy $\phi_t$ connecting $\phi$ to the identity. On the other hand, for a fixed contact isotopy $\phi_t$ Th\'{e}ret \cite{Th2} proved that the construction of Propositions \ref{ex_gf_ne} and \ref{ex_gf} does not depend, up to stabilization and fiber-preserving diffeomorphism, on the particular choice of subdivision of $\phi_t$ into $\mathcal{C}^1$-small pieces.} $F: \mathbb{R}^{2n}\times\mathbb{R}^{2N} \rightarrow \mathbb{R}$ for its lift $\Phi$ to $\mathbb{R}^{2n}$. Note that the property of being homogeneous of degree $2$ implies that all critical points of $F$ have critical value $0$. They come in $\mathbb{R}_+$-lines which correspond to $\mathbb{R}_+$-lines of fixed points of $\Phi$ and hence to discriminant points of $\phi$. Since $F: \mathbb{R}^{2n}\times\mathbb{R}^{2N} \rightarrow \mathbb{R}$ is homogeneous of degree $2$ it is determined by its restriction $f$ to the unit sphere $S^{2n+2N-1}$. $\mathbb{R}_+$-lines of critical points of $F$ correspond to critical points of $f$ with critical value $0$. From this it follows that discriminant points of $\phi$ correspond\footnote{It would be tempting to believe that critical points of $f$ correspond to translated points of $\phi$ (with critical value given by the amount of translation). However this cannot be true, for the following reason. We will see in Section \ref{RP} that if $\phi: S^{2n-1} \rightarrow S^{2n-1}$ lifts a contactomorphism $\overline{\phi}$ of $\mathbb{R}P^{2n-1}$ then its lift $\Phi$ to $\mathbb{R}^{2n}$ has a generating function $F: \mathbb{R}^{2n}\times\mathbb{R}^{2N} \rightarrow \mathbb{R}$ which is homogeneous of degree $2$ and invariant by the diagonal $\mathbb{Z}_2$-action, and hence induces a function $\overline{f}$ on $\mathbb{R}P^{2n+2N-1}$. Our $f:S^{2n+2N-1}\rightarrow\mathbb{R}$ is the lift of this function to $S^{2n+2N-1}$. If the critical points of $f$ where in 1-1 correspondence with the translated points of $\phi$ then we would also have a 1-1 correspondence between the critical points of $\overline{f}$ and the translated points of $\overline{\phi}$. Hence $\overline{\phi}$ would have at least $2n+2N$ translated points. If we now stabilize $F$ by taking the direct sum with a non-degenerate quadratic form on some extra $2N'$ coordinates then we obtain a new generating function $F': \mathbb{R}^{2n}\times\mathbb{R}^{2N+2N'} \rightarrow \mathbb{R}$ for $\Phi$. Note that the critical points of $F$ are in 1-1 correspondence with the critical points of $F'$. However the induced $\overline{f'}$ is a function on $\mathbb{R}P^{2n+2N+2N'-1}$, and hence it has at least $2n+2N+2N'$ critical points. Since $N'$ can be arbitrarily big, it would follow that $\overline{\phi}$ has infinitely many translated points. But as we saw in the introduction there are contactomorphisms of $\mathbb{R}P^{2n-1}$ with exactly $2n$ translated points.} to critical points of $f$ with critical value $0$.

\begin{lemma}
Non-degenerate critical points of $f$ with critical value $0$ correspond to non-degenerate discriminant points of $\phi$.
\end{lemma}

\begin{proof}
We will adapt to our situation the proof of Proposition 2.1 of \cite{C}. We first recall the notion of symplectic reduction, and how it can be use to interpret the construction of generating functions. Let $(W,\omega)$ be a symplectic manifold. A submanifold $Q$ of $W$ is called coisotropic (respectively isotropic) if at every point $q$ of $Q$ the symplectic orthogonal in $T_qW$ of $T_qQ$ is contained in (respectively contains) $T_qQ$. Given a coisotropic submanifold $Q$ of $(W,\omega)$, the distribution $\text{ker}(\omega|_{Q})$ is integrable and gives rise to a foliation of $Q$ by isotropic leaves. The quotient $W'$, if smooth, has a symplectic form $\omega'$ induced by $\omega$. The symplectic manifold $(W',\omega')$ is called the \textit{symplectic reduction} of $Q$. If $L$ is a Lagrangian submanifold of $W$ such that $L\cap Q$ is a submanifold and $T_q(L\cap Q)=T_qL \cap T_qQ$ for every $q\in L\cap Q$, then the quotient $L^{\text{red}}$ of $L\cap Q$ is a Lagrangian submanifold of $(W',\omega')$. We now apply this construction to our situation. Recall that we have a generating function $F:E\rightarrow \mathbb{R}$ defined on the total space of a fiber bundle $\pi: E\rightarrow B$. We consider the coisotropic submanifold $N_E$ of $T^{\ast}E$. Recall that $N_E$ is the fiber normal bundle of $\pi$. The isotropic leaves are the fibers of $E$, and thus its symplectic reduction is $T^{\ast}B$. The graph of $dF$ is a Lagrangian submanifold of $T^{\ast}E$, whose reduction is the Lagrangian submanifold of $T^{\ast}B$ generated by $F$. In our situation $F$ is defined on the total space of the fibration $\pi: \mathbb{R}^{2n}\times\mathbb{R}^{2N}\rightarrow\mathbb{R}^{2n}$, and it is a generating function for the Lagrangian submanifold $\Gamma_{\Phi}$ of $T^{\ast}\mathbb{R}^{2n}$ corresponding to the graph of $\Phi$. We denote by $L_1$ and $L_2$ respectively the graph of $dF$ and the 0-section in $T^{\ast}(\mathbb{R}^{2n}\times\mathbb{R}^{2N})$. Then the corresponding reduced Lagrangians $L_1^{\phantom{1}\text{red}}$ and $L_2^{\phantom{2}\text{red}}$ of $T^{\ast}\mathbb{R}^{2n}$ are respectively $\Gamma_{\Phi}$ and the 0-section. Since $\Phi$ is equivariant with respect to the $\mathbb{R}_{+}$-action on $\mathbb{R}^{2n}$, we have that $\Gamma_{\Phi}$ is equivariant with respect to the induced $\mathbb{R}_{+}$-action on $T^{\ast}\mathbb{R}^{2n}$. Consider a point $(q;\mu)\in\mathbb{R}^{2n}\times\mathbb{R}^{2N}$ lying on the unit sphere, and assume that it is a critical point of $F$ (equivalently, a critical point of $f$ with critical value $0$). Hence, $(q;\mu)$ belongs to $L_1\cap L_2$. Since $F$ is homogeneous of degree $2$, the whole line $l_{(q;\mu)}=\{\,(\lambda q; \lambda\mu)\,|\,\lambda\in\mathbb{R}_{+}\,\}$ is made of critical points of $F$ and hence belongs to $L_1\cap L_2$. It corresponds in the reduced space $\mathbb{R}^{2n}$ to a $\mathbb{R}_{+}$-line $l_{q}=\{\,\lambda q\,|\,\lambda\in\mathbb{R}_{+}\,\}$ in $L_1^{\phantom{1}\text{red}}\cap L_2^{\phantom{2}\text{red}}$. The point $(q;\mu)$ is a non-degenerate critical point of $f$ if and only if $T_{(q;\mu)}L_1\cap T_{(q;\mu)}L_2=T_{(q;\mu)}l_{(q;\mu)}$. But this condition is equivalent to the analogue condition in the reduced space, i.e. $T_q L_1^{\phantom{1}\text{red}}\cap T_q L_2^{\phantom{2}\text{red}}=T_q L_q$ which is satisfied if and only if $q$ is a non-degenerate discriminant point of $\phi$.
\end{proof}

We will now show monotonicity of generating functions. Recall that a contact isotopy $\phi_t$ of $S^{2n-1}$ is said to be \textit{positive} if it moves every point in a direction which is positively transverse to the contact distribution (or equivalently if it is generated by a positive contact Hamiltonian). This notion was introduced by Eliashberg and Polterovich in \cite{EP}.

\begin{lemma}\label{monotonicity}
If $\phi_t$, $t\in[0,1]$, is a positive contact isotopy of $S^{2n-1}$ then there is a 1-parameter family $F_t: \mathbb{R}^{2n}\times\mathbb{R}^{2N}\rightarrow\mathbb{R}$ of homogeneous generating functions for the lifted Hamiltonian isotopy $\Phi_t$  of $\mathbb{R}^{2n}$ such that $\frac{\partial F_t}{\partial t}>0$.
\end{lemma}

\begin{proof}
We first consider the case when $\Phi_t$ is $\mathcal{C}^1$-small for every $t \in [0,1]$ and thus has a homogeneous generating function $F_t: \mathbb{R}^{2n}\rightarrow\mathbb{R}$. Let $\Psi_{\Phi_t}$ be the Hamiltonian isotopy of $T^{\ast}\mathbb{R}^{2n}$ that makes the following diagram commute
$$
\xymatrix{
 \quad \overline{\mathbb{R}^{2n}}\times\mathbb{R}^{2n}\quad  \ar[r]^{\text{id}\times\Phi_t} \ar[d]_{\tau} &
 \quad \overline{\mathbb{R}^{2n}}\times\mathbb{R}^{2n} \quad \ar[d]^{\tau} \\
 \quad T^{\ast}\mathbb{R}^{2n} \quad \ar[r]_{\Psi_{\Phi_t}} &  \quad T^{\ast}\mathbb{R}^{2n}.}\quad
$$
Then the image of the 0-section by $\Psi_{\Phi_t}$ is the graph of $dF_t$. By the Hamilton-Jacobi equation \cite[\S 46]{A1} we know thus that
\begin{equation}\label{HJ}
\frac{\partial}{\partial t} F_t(q) = H_{\Psi_{\Phi_t}}\,\big(q,\frac{\partial F_t}{\partial q}\big)
\end{equation}
where $H_{\Psi_{\Phi_t}}: T^{\ast}\mathbb{R}^{2n}\rightarrow\mathbb{R}$ is the Hamiltonian generating $\Psi_{\Phi_t}$, $t \in [0,1]$. Note that $H_{\Psi_{\Phi_t}}=\overline{H_t}\circ \tau^{-1}$ where $\overline{H_t}: \overline{\mathbb{R}^{2n}}\times\mathbb{R}^{2n}\rightarrow\mathbb{R}$ is defined by $\overline{H_t}(z,z')=H_t(z')$. Here $H_t$, $t \in [0,1]$ is the Hamiltonian function of the isotopy $\Phi_t$ of $\mathbb{R}^{2n}$. Since $\phi_t$ is a positive contact isotopy we have thus that $H_{\Psi_{\Phi_t}}>0$, hence by (\ref{HJ}) also $\frac{\partial F_t}{\partial t}>0$. Consider now the case of a general positive contact isotopy $\phi_t$, $t \in [0,1]$. We take a subdivision $0=t_0<t_1<\cdots<t_k=1$ of $[0,1]$ into sufficiently many pieces, so that each $\phi_{t_j}\circ\phi_{t_{j-1}}^{\phantom{t_{j-1}}-1}$ is $\mathcal{C}^1$-small. Because of Lemma \ref{cf} the result then follows by induction from the $\mathcal{C}^1$-small case.
\end{proof}

\begin{rmk}\label{rmk_mon}
The previous lemma does not imply that if $\phi_t$ is a positive contact isotopy then there is a homogeneous generating function for the lift of $\phi_1$ which is positive. All we know is that there is a 1-parameter family $F_t: \mathbb{R}^{2n}\times\mathbb{R}^{2N}\rightarrow\mathbb{R}$ of homogeneous generating functions for the lifted $\Phi_t$ with $\frac{\partial F_t}{\partial t}>0$, and so in particular with $F_1>F_0$. $F_0$ is a generating function for the identity, but it is not necessarily the constant function $F_0=0$: it is in general a quadratic form on the fiber bundle $\mathbb{R}^{2n}\times\mathbb{R}^{2N}\rightarrow\mathbb{R}^{2n}$.
\end{rmk}

By the discussion above we see that we can detect discriminant points of a contactomorphism $\phi$  of $S^{2n-1}$ by looking at the preimage $f^{-1}(0)$ where $f$ is the restriction to the unit sphere of a generating function $F$ for the lifted $\Phi$: if $f^{-1}(0)$ is singular then there must be a discriminant point of $\phi$. Using this idea we can then also find the translated points of $\phi$, by looking for discriminant points of $a_t\circ\phi$ for $t\in[0,1]$ where as before $a_t$ is the negative Reeb flow of $S^{2n-1}$. We will do this in the next section. We will consider a 1-parameter family $F_t: \mathbb{R}^{2n}\times\mathbb{R}^{2N} \rightarrow \mathbb{R}$ of homogeneous generating functions for $a_t\circ\phi$, and look for values of $t$ for which $f_t^{\phantom{t}-1}(0)$ is singular, where $f_t$ is the restriction of $F_t$ to the unit sphere. A crucial element in the proof is given by the following proposition. Note first that $a_t$, $t\in[0,1]$, has a 1-parameter family of generating quadratic forms. This is true for any path of $\mathbb{C}$-linear symplectomorphisms of $\mathbb{R}^{2n}$, as can be seen by applying the proof of Proposition \ref{ex_gf} and using the observation that if $\phi$ is a $\mathcal{C}^1$-small $\mathbb{C}$-linear symplectomorphism of $\mathbb{R}^{2n}$ the $\Gamma_{\Phi}$ is the graph of the differential of a quadratic form.

\begin{prop}[\cite{Th1}]\label{mas}
If $A_t: \mathbb{R}^{2n}\times\mathbb{R}^{2N}\rightarrow\mathbb{R}$ is a 1-parameter family of generating quadratic forms for $a_t$, $t\in[0,1]$, then $\text{ind}(A_1)-\text{ind}(A_0)=2n$.
\end{prop}

We refer to \cite{Th1} for a proof of this proposition. We just notice that $2n$ is the Maslov index of the lift to $\mathbb{R}^{2n}$ of the path $t \mapsto a_t$, $t \in [0,1]$.

\section{Translated points for contactomorphisms of $S^{2n-1}$}\label{S}

Let $\phi$ be a contactomorphism of $S^{2n-1}$, contact isotopic to the identity. We will assume that all translated points of $\phi$ are non-degenerate. Under this assumption we will now prove that $\phi$ has at least 2 translated points.\\
\\
Let $a_t: S^{2n-1}\rightarrow S^{2n-1}$, $t\in[0,1]$, be the negative Reeb flow $a_t(z)=e^{-2\pi it}z$. The translated points of $\phi$ correspond to the union of the discriminant points of $a_t\circ\phi$ for $t$ varying in the interval $[0,1]$. By assumption, all discriminant points of $a_t\circ\phi$ are non-degenerate.\\
\\
As we know from the previous section, $\phi$ has a homogeneous generating function $F: \mathbb{R}^{2n}\times\mathbb{R}^{2N}\rightarrow\mathbb{R}$, and the contact isotopy $a_t$, $t\in[0,1]$, has a 1-parameter family of generating quadratic forms $A_t: \mathbb{R}^{2n}\times\mathbb{R}^{2N'}\rightarrow\mathbb{R}$. A 1-parameter family of homogeneous generating functions $F_t=A_t \sharp F: \mathbb{R}^{2n}\times (\mathbb{R}^{2n}\times\mathbb{R}^{2n}\times\mathbb{R}^{2N}\times\mathbb{R}^{2N'})\rightarrow\mathbb{R}$ for $a_t\circ\phi$, $t\in[0,1]$, is then obtained by applying the composition formula of Lemma \ref{cf}.\\
\\
We will denote by $f_t$, $t\in[0,1]$, the restriction of $F_t$ to the unit sphere $S^{2n+2M-1}$ of $\mathbb{R}^{2n+2M}$, where we set $M=2n+N+N'$. Then, as we already discussed, (non-degenerate) critical points of $f_t$ with critical value $0$ correspond to (non-degenerate) discriminant points of $a_t\circ\phi$, hence to (non-degenerate) translated points of $\phi$ that are translated by $t$. We need thus to detect values of $t$ for which $f_t^{\phantom{t}-1}(0)$ is singular. To do this, we will look at changes in the topology of the sublevel sets $N_t:=\{f_t\leq0\}\subset S^{2n+2M-1}$. \\
\\
Since $a_t$, $t\in[0,1]$, is a negative contact isotopy, By Lemma \ref{monotonicity} we have that $\frac{\partial F_t}{\partial t}<0$ and so the family of subsets $N_t$ of $S^{2n+2M-1}$ is increasing: $N_t\subset N_{t'}$ if $t<t'$. Consider the submanifold\footnote{It was proved by Th\'{e}ret \cite{Th2} that $0$ is a regular value of $f$ so that $V$ is indeed a smooth submanifold.} $V:=f^{-1}(0)$ of $[0,1]\times S^{2n+2M-1}$, where $f:[0,1]\times S^{2n+2M-1}\rightarrow\mathbb{R}$ denotes the total function $f(t,x)=f_t(x)$. Let $\pi:V\rightarrow[0,1]$ be the projection on the first factor. Then (non-degenerate) critical points of $f_t$ with critical value $0$ correspond to (non-degenerate) critical points of $\pi$ with critical value $t$. It follows from this that if for $t_0$, $t_1\in[0,1]$ the corresponding sublevel sets $N_{t_0}$, $N_{t_1}$ have different homotopy type then there must be a value of $t$ in $[t_0,t_1]$ for which $f_t^{\phantom{t}-1}(0)$ is singular, hence there must be a translated point of $\phi$ which is translated by $t$. If instead we cross a (non-degenerate) critical point $x$ of index $k$ then $N_t$ changes by attaching a $k$-cell.\\
\\
The idea of the proof is now to compare the homotopy types of $N_0$ and $N_1$, and to show that the change in the topology from $N_0$ to $N_1$ can only be obtained by crossing at least 2 critical points. \\
\\
Note that $A_0$ and $A_1$ are quadratic generating forms of the identity. Th\'{e}ret \cite{Th1} proved that in such a situation we can find smooth paths $\Psi_{s}^{\phantom{s}0}$, $\Psi_{s}^{\phantom{s}1}$, $s\in[0,1]$, of diffeomorphisms of $\mathbb{R}^{2n}\times\mathbb{R}^{2M}$ (that we can assume to be equivariant with respect to the $\mathbb{R}_+$-action) such that $\Psi_{0}^{\phantom{0}0}=\Psi_{0}^{\phantom{0}1}=\text{id}$, and $A_0\circ\Psi_{1}^{\phantom{1}0}=\overline{A_0}$, $A_1\circ\Psi_{1}^{\phantom{1}1}=\overline{A_1}$ for non-degenerate quadratic forms $\overline{A_0}$, $\overline{A_1}:\mathbb{R}^{2M}\rightarrow\mathbb{R}$. Then $(A_0\circ\Psi_{s}^{\phantom{s}0})\sharp F$ is a family of generating functions for $\phi$. By Lemma 4.8 in \cite{Th2} these generating functions are all equivalent. Thus $N_0$ has the same homotopy type of $\{\overline{A_0}\sharp F\leq0\}$. Similarly, $N_1$ has the same homotopy type of $\{\overline{A_1}\sharp F\leq0\}$.\\
\\
If the function $F$ happens to be positive then the rest of the proof is particularly simple. In this case we have indeed that $N_0\simeq \{\overline{A_0}\sharp F\leq0\}\simeq\{\overline{A_0}\leq0\}\simeq S^{i(\overline{A_0})-1}$ and similarly $N_1\simeq S^{i(\overline{A_1})-1}$. Since, by Proposition \ref{mas}, $i(\overline{A_1})-i(\overline{A_0})=2n$ and since all critical points are assumed to be non-degenerate, elementary arguments of Morse theory allows us to conclude that either there are two different values of $t$ for which $f_t^{\phantom{t}-1}(0)$ is singular or a single value of $t$ for which $f_t^{\phantom{t}-1}(0)$ has a positive dimensional singular set. In either case there are at least $2$ translated points of $\phi$.\\
\\
In general, even though we can always assume without loss of generality that $\phi$ is the time-1 map of a positive contact isotopy (we can just compose any contact isotopy connecting it to the identity with sufficiently many iterations of the Reeb flow), its generating function $F$ is not necessarily positive (see Remark \ref{rmk_mon}).\\
\\
The proof in the general case (i.e. if $F$ is not necessarily positive) can be completed as follows. We have
$$
N_0\simeq \{\overline{A_0}\sharp F\leq0\}\simeq\{\overline{A_0}\leq0\}\ast\{F\leq0\}\simeq S^{i(\overline{A_0})-1}\ast\{F\leq0\}
$$
and similarly $N_1\simeq  S^{i(\overline{A_1})-1}\ast\{F\leq0\}$. Here $\ast$ denotes the join operator. Recall that the join $X\ast Y$ of two topological spaces $X$ and $Y$ is defined to be the quotient of $X\times Y\times[0,1]$ by the identifications $(x,y_1,0)\sim(x,y_2,0)$ and $(x_1,y,1)\sim(x_2,y,1)$. Recall also that $X\ast Y \simeq\Sigma(X\wedge Y)$ where $X\wedge Y$ is the smash product of $X$ and $Y$, i.e. the quotient of the product $X\times Y$ by the wedge $X\vee Y$, and $\Sigma$ the suspension operator. Since the homology of $\Sigma(X\wedge Y)$ is the homology of $X\wedge Y$ shifted by 1, and the homology of $X\wedge Y$ can be computed by the relative K\"{u}nneth formula (see for example \cite[page 276]{Hat}) we can then conclude by the same arguments as before.

\section{Translated points for contactomorphisms of $\mathbb{R}P^{2n-1}$}\label{RP}

We want to prove in this section that every contactomorphism of $\mathbb{R}P^{2n-1}$ which is contact isotopic to the identity has at least $2n$ translated points. The proof is analogous to Th\'{e}ret's proof \cite{Th2} of the existence of $n$ \textit{rotation numbers} for a Hamiltonian isotopy of $\mathbb{C}P^{n-1}$. As we will see, the non-degeneracy condition for translated points is not needed here.\\
\\
We consider the standard contact form on $\mathbb{R}P^{2n-1}$, i.e. the one obtained by quotienting the standard contact form of $S^{2n-1}$ by the antipodal action of $\mathbb{Z}_2$. Let $\phi$ be a contactomorphism of $\mathbb{R}P^{2n-1}$. To study its translated points we will lift it to a ($\mathbb{Z}_2$-equivariant) contactomorphism of $S^{2n-1}$. Note that $\phi$ can be lifted to $S^{2n-1}$ in two different ways. For example if $\phi$ is the identity then its lift is either the identity of $S^{2n-1}$ or the antipodal map. However if $\phi$ is the time-1 map of a contact isotopy $\phi_t$ then we can uniquely define its lift $\widetilde{\phi}$ to be the time-1 map of the contact isotopy $\widetilde{\phi_t}$ of $S^{2n-1}$ that lifts $\phi_t$ and begins at the identity. Note that $\widetilde{\phi_t}$ is generated by the Hamiltonian function $\widetilde{h_t}=h_t\circ \pi$ where $\pi: S^{2n-1} \rightarrow \mathbb{R}P^{2n-1}$ is the projection and $h_t: \mathbb{R}P^{2n-1}\rightarrow\mathbb{R}$ is the Hamiltonian function of $\phi_t$. We now lift the contactomorphism $\widetilde{\phi}$ of $S^{2n-1}$ to a $\mathbb{R}_+$-equivariant Hamiltonian symplectomorphism $\Phi$ of $\mathbb{R}^{2n}$, as described in Section \ref{bg}. Recall from Section \ref{gf} that $\Phi$ has a generating function which is homogeneous of degree $2$. We will now show that, since $\widetilde{\phi}$ is the lift of a contactomorphism of $\mathbb{R}P^{2n-1}$, this generating function can be assumed to be also $\mathbb{Z}_2$-invariant.

\begin{prop}
Let $\phi$ be a contactomorphism of $\mathbb{R}P^{2n-1}$ contact isotopic to the identity, and consider its lift $\Phi$ to $\mathbb{R}^{2n}$. Then $\Phi$ has a generating function $F: \mathbb{R}^{2n}\times\mathbb{R}^{2N} \rightarrow \mathbb{R}$ which is homogeneous of degree $2$ and invariant by the diagonal action of $\mathbb{Z}_2$ on $\mathbb{R}^{2n}\times\mathbb{R}^{2N}$. In other words, $F$ is \emph{conical} i.e. $F(\lambda q; \lambda\mu) = \lambda^2 F(q;\mu)$ for all $\lambda\in\mathbb{R}$.
\end{prop}

\begin{proof}
Consider the homogeneous generating function constructed in Proposition \ref{ex_gf}. Since the composition formula of Lemma \ref{cf} preserves the property of being $\mathbb{Z}_2$-invariant, it is enough to prove that if $\Phi$ is $\mathcal{C}^1$-small then its homogeneous generating function
$F: \mathbb{R}^{2n}\rightarrow \mathbb{R}$ is $\mathbb{Z}_2$-invariant. We denote by $\sigma: \mathbb{R}^{2n}\rightarrow\mathbb{R}$ the map $\sigma(x,y)=(-x,-y)$ and use the same notation also for the induced map on $T^{\ast}\mathbb{R}^{2n}$. Since $\Phi$ is $\mathbb{Z}_2$-equivariant we have that $\sigma(\Gamma_{\Phi})=\Gamma_{\Phi}$. Consider the function $\overline{F}: \mathbb{R}^{2n}\rightarrow \mathbb{R}$, $\overline{F}=F\circ\sigma$. Then $\overline{F}$ generates $\sigma(\Gamma_{\Phi})=\Gamma_{\Phi}$. Recall that the generating function $F: \mathbb{R}^{2n}\rightarrow \mathbb{R}$ of $\Phi$ is uniquely defined once we normalize it by $F(0)=0$ (this normalization is automatic in our case since $F$ is homogeneous of degree $2$). Since $\overline{F}(0)=F(0)=0$ we have thus that $\overline{F}=F$, hence $F$ is $\mathbb{Z}_2$-invariant.
\end{proof}

To find translated points of the contactomorphism $\widetilde{\phi}$ of $S^{2n-1}$ we will now apply the same strategy as in the previous section. However, as we will now see, the fact that the generating function is $\mathbb{Z}_2$-invariant will imply that in this case the translated points are at least $2n$ instead of only $2$. \\
\\
As before let $a_t: S^{2n-1} \rightarrow S^{2n-1}$ be the negative Reeb flow $z \mapsto e^{-2\pi i t}z$, and consider the contact isotopy $a_t\circ\widetilde{\phi}$, $t\in[0,1]$. Then
$F_t=A_t \sharp F: \mathbb{R}^{2n}\times\mathbb{R}^{2M}\rightarrow\mathbb{R}$ is a path of conical generating functions for $a_t\circ\phi$, $t\in[0,1]$. Since every $F_t$ is conical, we get induced functions $f_t: \mathbb{R}P^{2n+2M-1}\rightarrow\mathbb{R}$, $t\in[0,1]$.\\
\\
We are interested in the translated points of $\phi:\mathbb{R}P^{2n-1}\rightarrow\mathbb{R}P^{2n-1}$. Note that their number is half the number of translated points of $\widetilde{\phi}$, and so half the number of discriminant points of $a_t\circ\widetilde{\phi}$ for $t$ varying in the interval $[0,1]$. Since discriminant points of $a_t\circ\widetilde{\phi}$ correspond to critical points of the restriction of the generating function $F_t$ to the unit sphere, which are twice as many as the critical points of $f_t$, we have that translated points of $\phi$ correspond to the union for all $t\in[0,1]$ of the critical points of $f_t$.\\
\\
Recall that all critical points of $f_t$ have critical value $0$. To detect their presence we will proceed as in the previous section, i.e. we will look at the changes in the topology of the subset $N_t:=\{f_t\leq0\}$ of $\mathbb{R}P^{2n+2M-1}$ for $t\in[0,1]$. In order to detect the changes in the topology of the $N_t$ we will use the cohomological index for subsets of projective spaces introduced by Fadell and Rabinowitz \cite{FR} (see also \cite{Giv} and \cite{Th2}). \\
\\
The cohomological index of a subset $X$ of a real projective space $\mathbb{R}P^m$ is defined as follows. Recall that $H^{\ast}(\mathbb{R}P^m;\mathbb{Z}_2)=\mathbb{Z}_2[u]/u^{m+1}$ where $u$ is the generator of $H^1(\mathbb{R}P^m;\mathbb{Z}_2)$. We define
$$
\text{ind}(X)=1+\text{max}\{\,k\in\mathbb{N}\;|\;i_X^{\phantom{X}\ast}(u^k)\neq 0\,\}
$$
where $i_X: X\rightarrow\mathbb{R}P^m$ is the inclusion. We also set by definition $\text{ind}(\emptyset)=0$. \\
\\
Given a conical function $F:\mathbb{R}^{m+1}\rightarrow\mathbb{R}$ we will denote by $\text{ind}(F)$ the index of the sublevel set at $0$ of the induced function on $\mathbb{R}P^m$. Note that if $Q$ is a quadratic form on $\mathbb{R}^{m+1}$ then $\text{ind}(Q)=i(Q)+\text{dim}(\text{ker}(Q))$, where $i(Q)$ denotes the index of $Q$ as a quadratic form.

\begin{lemma}\label{giv}
Let $F$ and $G$ be functions defined on $\mathbb{R}^{m+1}$ and $\mathbb{R}^{m'+1}$ respectively, and consider the function $F\oplus G: \mathbb{R}^{m+m'+2}\rightarrow\mathbb{R}$. Then we have
$$
\text{ind} \big(F\oplus G\big) = \text{ind}(F) + \text{ind}\big(G).
$$
\end{lemma}

We refer for a proof to \cite{Giv}[Appendices A and B]. The idea is the following. Let $f:\mathbb{R}P^m\rightarrow\mathbb{R}$, $g:\mathbb{R}P^{m'}\rightarrow\mathbb{R}$ and $f\oplus g:\mathbb{R}P^{m+m'+1}\rightarrow\mathbb{R}$ be the functions induced by $F$, $G$ and $F\oplus G$ respectively. Then $\{f\oplus g\leq 0\}\subset\mathbb{R}P^{m+m'+1}$ is homotopy equivalent to the projective join of $\{f\leq 0\}\subset\mathbb{R}P^m$ and  $\{g\leq 0\}\subset\mathbb{R}P^{m'}$. Recall that the projective join $A\ast B$ of a subset $A$ of $\mathbb{R}P^a$ and a subset $B$ of $\mathbb{R}P^b$ is the union of all the projective lines in $\mathbb{R}P^{a+b+1}$ passing through a point of $A$ and a point of $B$, where we regard $\mathbb{R}P^a$ and $\mathbb{R}P^b$ as subspaces of $\mathbb{R}P^{a+b+1}$. Using an equivariant K\"{u}nnet formula one can express the $\mathbb{Z}_2$-cohomology of a join $A\ast B$ in terms of the $\mathbb{Z}_2$-cohomology of $A$ and $B$. This fact can then be used to prove that $\text{ind}(A\ast B)=\text{ind}(A)+\text{ind}(B)$.\\
\\
We can now complete the proof of the existence of $2n$ translated points for a contactomorphism $\phi$ of $\mathbb{R}P^{2n-1}$ which is contact isotopic to the identity. Recall that we consider a 1-parameter family $F_t=A_t\sharp F: \mathbb{R}^{2n} \times \mathbb{R}^{2M}\rightarrow \mathbb{R}$ of generating functions for the isotopy $a_t \circ \phi$, $t\in [0,1]$, and the induced 1-parameter family of functions $f_t: \mathbb{R}P^{2n+2M-1}\rightarrow\mathbb{R}$. For $t\in[0,1]$ we set $l(t)=\text{ind}(N_t)$ where $N_t=\{f_t\leq 0\}$. By Lemma \ref{monotonicity} we have that the map $t\mapsto l(t)$ is increasing. Moreover it is locally constant in the neighborhood of a value of $t$ for which $f_t^{\phantom{t}-1}(0)$ is non-singular.\\
\\
The following lemma is proved in \cite{Th1}. We rewrite the proof here for the sake of the reader.

\begin{lemma}
If $l(t_0^{\phantom{0}+})=l(t_0^{\phantom{0}-})+i$ then the set of critical points of $f_{t_0}$ has index greater or equal to $i$, and so in particular it is infinite if $i\geq 2$.
\end{lemma}

\begin{proof}
We will use the following two properties of the cohomological index.
\begin{itemize}
 \item[-] (\textit{Continuity}) Let $X$ be a subset of $\mathbb{R}P^m$. Then there is a closed neighborhood $\mathcal{U}$ of $X$ such that $\text{ind}(X)=\text{ind}(\mathcal{U})$.
 \item[-] (\textit{Sub-additivity}) Let $X$, $X'$ be two subsets of $\mathbb{R}P^m$. Then $\text{ind}(X\cup X')\leq\text{ind}(X)+\text{ind}(X')$.
\end{itemize}
We refer to \cite{FR} for a proof of these properties. Now let $K_{t_0}$ be the set of critical points of $f_{t_0}$ with critical value $0$. By continuity we have a closed neighborhood $\mathcal{U}$ of $K_{t_0}$ such that $\text{ind}(\mathcal{U})=\text{ind}(K_{t_0})$. On the other hand we can deform $N_{t_0^{\phantom{0}+}}$ to $N_{t_0^{\phantom{0}-}}\cup \mathcal{U}$, so that $\text{ind}(N_{t_0^{\phantom{0}+}})=\text{ind}(N_{t_0^{\phantom{0}-}}\cup \mathcal{U})$. By sub-additivity we then have $\text{ind}(N_{t_0^{\phantom{0}+}})\leq \text{ind}(N_{t_0^{\phantom{0}-}}) + \text{ind}(\mathcal{U})$. Since we assume that $\text{ind}(N_{t_0^{\phantom{0}+}})=\text{ind}(N_{t_0^{\phantom{0}-}})+i$ we conclude that $\text{ind}(K_{t_0})\geq i$, as we wanted.
\end{proof}

By this discussion, our result follows if we prove that $l(1)-l(0)=2n$.\\
\\
Recall that $l(1)=\text{ind}(A_1\sharp F)$ and $l(0)=\text{ind}(A_0\sharp F)$. As in the previous section we have smooth paths $\Psi^0_s$ and $\Psi^1_s$, $s\in[0,1]$ of fiber preserving diffeomorphisms of $\mathbb{R}^{2n} \times \mathbb{R}^{2M}$ which are equivariant with respect to the diagonal actions of $\mathbb{R}_+$ and $\mathbb{Z}_2$ and are such that $\Psi^0_0$ and $\Psi^1_0$ are the identity and $A_0\circ \Psi^0_1=\overline{A_0}$ and $A_1\circ \Psi^1_1=\overline{A_1}$ for some non-degenerate quadratic forms $\overline{A_0}$ and $\overline{A_1}$ on $\mathbb{R}^{2M}$. We will now use this fact to apply Lemma \ref{giv} to our situation.\\
\\
We have

\begin{eqnarray*}
l(1)-l(0)&=&\text{ind}\big(A_1\sharp F\big)-\text{ind}\big(A_0\sharp F\big)\\
&=&\text{ind}\big(\overline{A_1}\oplus F\big)-\text{ind}\big(\overline{A_0}\oplus F\big)\\
&=&\text{ind}\big(\overline{A_1}\big) + \text{ind}(F) - \text{ind}\big(\overline{A_0}\big) - \text{ind}(F) \\
&=&\text{ind}(\overline{A_1}) - \text{ind}(\overline{A_0}).
\end{eqnarray*}

But as we already mentioned the cohomological index of a non-degenerate quadratic form  is equal to its index as a quadratic form, thus
$$
l(1)-l(0) = i(\overline{A_1}) - i(\overline{A_0}) = 2n.
$$
where the last equality follows from Proposition \ref{mas}.

\end{document}